\documentclass[12pt]{amsart}

\usepackage{graphics}
\usepackage{amsmath}
\usepackage{amsfonts}
\usepackage{amssymb}
\usepackage{amscd}

\newcommand{\C}{{\mathbb C}}
\renewcommand{\P}{{\mathbb P}}
\newcommand{\T}{{\mathbb T}}
\newcommand{\D}{{\mathbb D}}
\newcommand{\A}{{\mathcal A}}
\newcommand{\B}{{\mathcal B}}
\renewcommand{\O}{{\mathcal O}}
\newcommand{\R}{{\mathbb R}}

\newtheorem{theorem}{\bf Theorem}
\newtheorem{lemma}{\bf Lemma}
\newtheorem{proposition}{\bf Proposition}
\newtheorem{corollary}{\bf Corollary}

\title{Siciak-Zahariuta extremal functions, \\ analytic discs and polynomial hulls}

\author{Finnur L\'arusson}
\address{School of Mathematical Sciences, University of Adelaide, Adelaide SA 5005, Australia.} 
\email{finnur.larusson@adelaide.edu.au}

\author{Ragnar Sigurdsson}
\address{Science Institute, University of Iceland, Dunhaga 3, IS-107 Reykjav\'ik, Iceland.} 
\email{ragnar@hi.is}


\date{23 August 2008}

\begin{document}

\begin{abstract}  We prove two disc formulas for the Siciak-Zahariuta extremal function of an arbitrary open subset of complex affine space.  We use these formulas to characterize the polynomial hull of an arbitrary compact subset of complex affine space in terms of analytic discs.  Similar results in previous work of ours required the subsets to be connected.
\end{abstract}

\maketitle

\tableofcontents

\section{Introduction}

\noindent
The Siciak-Zahariuta extremal function $V_X$ of a subset $X$ of complex affine space $\C^n$ is defined as the supremum of all entire plurisubharmonic functions $u$ of minimal growth with $u|X\leq 0$.  It is also called the pluricomplex Green function of $X$ with logarithmic growth or a logarithmic pole at infinity (although this is a bit of a misnomer if $X$ is not bounded).  A plurisubharmonic function $u$ on $\C^n$ is said to have minimal growth, and belong to the Lelong class $\mathcal L$, if $u-\log^+\|\cdot\|$ is bounded above on $\C^n$.  If $X$ is open and nonempty, then $V_X\in \mathcal L$.  More generally, if $X$ is not pluripolar, then the upper semicontinuous regularization $V_X^*$ of $V_X$ is in $\mathcal L$, and if $X$ is pluripolar, then $V_X^*=\infty$.  Siciak-Zahariuta extremal functions play a fundamental role in pluripotential theory and have found important applications in approximation theory, complex dynamics, and even in arithmetic geometry.

The theory of disc functionals was founded in the late 1980s by Poletsky.  It provides an alternative approach to the extremal functions of pluripotential theory, usually defined as suprema of certain classes of plurisubharmonic functions, by realizing them as envelopes of disc functionals.  By now, disc formulas have been proved for just about all the commonly used plurisubharmonic extremal functions, such as largest plurisubharmonic minorants, including relative extremal functions, and pluricomplex Green functions of various kinds.  This project is advanced here with proofs of the first disc formulas for the Siciak-Zahariuta extremal function of an arbitrary open subset of $\C^n$.  

In previous work \cite{LarussonSigurdsson1}, we treated the connected case, motivated by Lempert's disc formula for the convex case (\cite{Momm}, Appendix).  The first of our two disc formulas is analogous to the formula in \cite{LarussonSigurdsson1}, but has a modified functional and a larger class of discs.  In \cite{LarussonSigurdsson1}, we showed that for every connected open subset $X$ of $\C^n$, the Siciak-Zahariuta extremal function of $X$ is the envelope of the disc functional $J$ with
$$J(f)=-\sum_{\zeta\in f^{-1}(H_\infty)} m_f(\zeta)\log|\zeta|,$$
where $f$ is an analytic disc in $\P^n$ that intersects $H_\infty$ with multiplicity $m_f(\zeta)$ at $\zeta\in\D$, with respect to the class $\A_{\P^n}^X$ of closed analytic discs $f$ in $\P^n$ with $f(0)\in\C^n$ and $f(\T)\subset X$.  (For notation and terminology, see the end of this introduction.)  This result fails for general open sets $X$, which is no wonder, since a closed analytic disc mapping $\T$ into $X$ must have all its boundary values in a single connected component of $X$.  We therefore extend $\A_{\P^n}^X$ to the class $\B_{\P^n}^X$ of analytic discs $f$ in $\P^n$ with $f(0)\in\C^n$ that have a bounded holomorphic lifting $\tilde f:\D\to\C_*^{n+1}$ such that the boundary map $f^*=\pi\circ\tilde f^*$ takes a conull subset of $\T$ into $X$.  
The envelope of $J$ with respect to $\B_{\P^n}^X$ is too small: $V_X\neq E_{\B_{\P^n}^X}J$ for example when $X$ is a disc in the plane (see Section 3).  It turns out to be appropriate to replace $J$ by the functional $I$ with
$$I(f)=-\max\nolimits_{\tilde f}\log|g(0)|,$$
where $\tilde f$ runs through all bounded holomorphic liftings of $f$, and $g$ is the inner factor of the zeroth component of $\tilde f$.  Now $J(f)=-\log|B(0)|$, where $B$ is the Blaschke factor of the zeroth component of every $\tilde f$, so $J\leq I$.

For an arbitrary open subset $X$ of $\C^n$, the fundamental inequality $V_X\leq E_{\B_{\P^n}^X}I$, justifying the choice of $\B_{\P^n}^X$ and $I$, is proved in Proposition \ref{inequality}.  Our main result, Theorem \ref{maintheorem}, is the opposite inequality, giving the disc formula
$$V_X(z)=\inf\{I(f):f\in\B_{\P^n}^X, f(0)=z\}$$
for every $z\in\C^n$.  To prove Theorem \ref{maintheorem}, as in the proof in \cite{LarussonSigurdsson1} of the analogous formula $V_X=E_{\A_{\P^n}^X}J$ for a connected open set $X$, we use our idea of a \lq\lq good\rq\rq\ auxiliary class of analytic discs, but the class is now different.  The proof of the key lemma, here Lemma \ref{mainlemma}, is also different and uses a construction due to Bu and Schachermayer \cite{BuSchachermayer}.  

The second formula, in Theorem \ref{secondformula}, uses a smaller class of discs, for which the new functional $I$ has an alternative formulation.  Let $f=(f_1,\dots,f_n):\D\to\C^n$ be a Nevanlinna disc, that is, an analytic disc whose components lie in the Nevanlinna class $N$.  Factor each component $f_j$ as as $B_j h_j s_j /t_j$, where $B_j$ is a Blaschke product, $h_j$ is an outer function, and $s_j$, $t_j$ are singular functions given by mutually singular measures.  We define the {\it negative mass} $\nu(f)$ of $f$ to be the mass of the smallest Borel measure on $\T$ that is no smaller than the singular measures corresponding to $t_1,\dots,t_n$.  For an open subset $X$ of $\C^n$, we denote by $\mathcal N_X^-$ the set of Nevanlinna discs in $\C^n$ whose boundary map takes a conull subset of $\T$ into $X$, and whose components have $s_j=1$.  Our second disc formula is derived from the first one, and states that the Siciak-Zahariuta extremal function of $X$ is the envelope of $\nu$ with respect to the class $\mathcal N_X^-$.  More explicitly,
$$V_X(z)=\inf\{\nu(f):f\in\mathcal N_X^-, f(0)=z\}$$
for every $z\in\C^n$.

The first major application of the theory of disc functionals was Poletsky's characterization of the polynomial hull $\hat K$ of a compact subset $K$ of $\C^n$ (\cite{Poletsky2}; see also \cite{LarussonSigurdsson0}, Theorem 7.4).  One version of Poletsky's result states that if $B$ is an open ball containing $K$, then $a\in B$ is in $\hat K$ if and only if for every neighbourhood $U$ of $K$ and every $\epsilon>0$, there is a closed analytic disc $f$ in $B$ with $f(0)=a$ and $\sigma(\T\setminus f^{-1}(U))<\epsilon$.

It is an easy consequence of our second disc formula that $a\in\C^n$ is in $\hat K$ if and only if for every neighbourhood $U$ of $K$ and every $\epsilon>0$, there is $f\in\mathcal N_U^-$ with $f(0)=a$ and $\nu(f)<\epsilon$ (Corollary \ref{hulls}).  In \cite{LarussonSigurdsson2}, we proved this under the assumption that $K$ is connected.  Our characterization and Poletsky's complement each other; neither implies the other in any obvious way.

\medskip

Let us list some definitions and notation used in the paper.  For $r>0$, we let $D_r$ denote the open disc $\{z\in\C:|z|<r\}$ with boundary $T_r$.  We set $\D=D_1$.  The unit circle $\T=T_1$ carries the normalized arc length measure $\sigma$.  A bounded holomorphic function, and more generally a Nevanlinna function, $f:\D\to\C$ has a nontangential boundary map $f^*:\T\to\C$, defined almost everywhere on $\T$ and measurable.  We view complex affine space $\C^n$ as the subset of complex projective space $\P^n$ with projective coordinates $[z_0,\dots,z_n]$ where $z_0\neq 0$, and denote by $H_\infty\subset\P^n$ the hyperplane at infinity where $z_0=0$.  The Euclidean norm on $\C^n$ is denoted $\|\cdot\|$.  We write $\C_*^{n+1}=\C^{n+1}\setminus\{0\}$ and let $\pi$ denote the projection $\C_*^{n+1}\to\P^n$.  We set $\lambda(z)=\log|z_0|$ for $z=(z_0,\dots,z_n)\in\C^{n+1}$.

An analytic disc in a complex manifold $Y$ is a holomorphic map $\D\to Y$.  We denote by $\O(\D,Y)$ the set of all analytic discs in $Y$.  A closed analytic disc in $Y$ is a map from the closed unit disc $\overline\D$ into $Y$ that extends to a holomorphic map on a neighbourhood of $\overline\D$.  We denote by $\O(\overline\D,Y)$ the set of all closed analytic discs in $Y$.  A disc functional on $Y$ is a map $\A\to[-\infty,\infty]$, where $\A\subset\O(\D,Y)$.  The envelope of $H$ (with respect to $\A$) is the function $E_\A H:Y\to[-\infty,\infty]$ with
$$E_\A H(y)=\inf\{H(f):f\in\A, f(0)=y\}.$$

\section{Two disc functionals}

\noindent
In \cite{LarussonSigurdsson1}, we introduced the disc functional $J$ on $\P^n$ defined by the formula
$$J(f)=-\sum_{\zeta\in f^{-1}(H_\infty)} m_{f_0}(\zeta)\log|\zeta|\ \geq \ 0, \qquad f\in\O(\D,\P^n).$$
Here, $m_{f_0}(\zeta)$ denotes the multiplicity at $\zeta$ of the intersection of $f$ with $H_\infty$, in other words, the order of the zero of the component $f_0$ at $\zeta$ when $f$ is expressed as $[f_0,\dots,f_n]$ in projective coordinates, that is, lifted to a holomorphic map $(f_0,\dots,f_n):\D\to\C_*^{n+1}$.  When the zeros of $f_0$ are not isolated, that is, $f(\D)\subset H_\infty$, we set $J(f)=\infty$, and when $f(\D)\cap H_\infty=\varnothing$, we set $J(f)=0$.  

We see that $J(f)$ is finite if and only if $f(0)\in\C^n$ and the zeros of $f_0$, counted with multiplicities, satisfy the Blaschke condition.  Then we can factor $f_0$ as $Bh$, where $B$ is a Blaschke product and $h$ is a zero-free holomorphic function on $\D$.  The Blaschke product $B$ is uniquely determined by $f$ up to multiplication by a number in $\T$, and 
$$J(f)=-\log|B(0)|.$$
Thus $f\in\O(\D,\P^n)$ has $J(f)<\infty$ if and only if $f(0)\in\C^n$ and $f$ has a lifting $\D\to\C_*^{n+1}$ whose zeroth component is a Blaschke product.  This lifting is then unique up to multiplication by a number in $\T$.

If $f\in\O(\D,\P^n)$ and $\tilde f=(f_0,\dots,f_n)$ is any holomorphic lifting of $f$, then the Riesz Representation Theorem applied to the subharmonic function $\lambda\circ\tilde f=\log|f_0|$ gives
$$\lambda(\tilde f(0)) = \int_{T_r}\lambda\circ\tilde f\,d\sigma + \frac 1{2\pi}\int_{D_r} \log|\cdot|\,\Delta(\lambda\circ\tilde f)$$
for each $r\in (0,1)$.  Also,
$$\frac 1{2\pi}\int_\D \log|\cdot|\,\Delta(\lambda\circ\tilde f) = \sum_{\zeta\in f_0^{-1}(0)} m_{f_0}(\zeta)\log|\zeta|=-J(f),$$
so
$$J(f)=\lim_{r\to1-}\int_{T_r}\lambda\circ\tilde f\,d\sigma-\lambda(\tilde f(0)).$$
The terms on the right-hand side are finite if and only if $J(f)<\infty$.

Let $X$ be an open subset of $\C^n$.  We denote by $\A_{\P^n}^X$ the set of closed analytic discs $f$ in $\P^n$ with $f(0)\in\C^n$ and $f(\T)\subset X$.  In \cite{LarussonSigurdsson1}, we proved that if $X$ is connected, then the Siciak-Zahariuta extremal function $V_X$ of $X$ satisfies the disc formula
$$V_X=E_{\A_{\P^n}^X}J.$$
We also noted that this is usually false if $X$ is not connected.  While we always have $V_X\leq E_{\A_{\P^n}^X}J$, there are simple examples for which equality fails.  Indeed, by our result, $E_{\A_{\P^n}^X}J$ is the infimum of the Siciak-Zahariuta extremal functions of the connected components of $X$, and this infimum will not be plurisubharmonic in general.

To extend the disc formula to disconnected sets, a sensible strategy would be to replace $\A_{\P^n}^X$ by a larger set $\B$ of analytic discs still satisfying the inequality $V_X\leq E_{\B}J$.  The discs in $\B$ should have boundary values in $X$ in some weak sense.  Also, it is natural to require $\B$ to be invariant under precomposition by automorphisms of $\D$.  The inequality $V_X\leq E_{\B}J$ then turns out to impose a strong restriction on $\B$, made explicit by the following result.

\begin{proposition}
\label{classforJ}
Let $X$ be a nonempty bounded open subset of $\C^n$.  Let $f\in\O(\D,\P^n)$ with $J(f)<\infty$ have a holomorphic lifting $\tilde f=(f_0,\dots,f_n):\D\to\C_*^{n+1}$ such that $f_0,\dots,f_n$ have nontangential limits at almost every point of $\T$ and the boundary map $f^*=\pi\circ\tilde f^*$ takes a conull subset of $\T$ into $X$.  Then the following are equivalent.
\begin{enumerate}
\item[(i)] For every $\alpha\in\text{\rm Aut }\D$ with $f(\alpha(0))\in\C^n$,
$$V_X((f\circ\alpha)(0)) \leq J(f\circ\alpha).$$
\item[(ii)]  If $(B,g_1,\dots,g_n)$ is a lifting of $f$ whose zeroth component $B$ is a Blaschke product, then the other components $g_1,\dots,g_n$ are bounded on $\D$.
\end{enumerate}
\end{proposition}

\begin{proof}  (i) $\Rightarrow$ (ii):  By assumption,
$$(V_X\circ f + \log|B|)(\alpha(0)) = V_X((f\circ\alpha)(0))-J(f\circ\alpha)\leq 0$$
for all $\alpha\in\text{Aut }\D$ with $f(\alpha(0))\in\C^n$.  Being bounded, $X$ is contained in a ball $K$ centred at the origin.  Then $V_X\geq V_K:z\mapsto \log^+\|cz\|$ for some constant $c>0$.  Thus, on $\D\setminus f^{-1}(H_\infty)$, we have $\log^+\|cf\|+\log|B|\leq 0$, so $\|cBf\|\leq 1$.  Now $Bf=(g_1,\dots,g_n)$, so $\|(g_1,\dots,g_n)\|\leq 1/c$ on $\D\setminus f^{-1}(H_\infty)$ and hence on $\D$.

(ii) $\Rightarrow$ (i):  We may assume that $\alpha$ is the identity.  Consider the subharmonic function $u=V_X\circ f+\log|B|$ on $\D\setminus f^{-1}(H_\infty)$.  Take a ball $K$ inside $X$.  We may assume that $K$ is centred at the origin.  Then $V_X\leq V_K:z\mapsto \log^+\|cz\|$ for some constant $c>0$.  Thus, 
$$u\leq \log^+\|cf\|+\log|B|=\max\{\log|B|, \log\|c(g_1,\dots,g_n)\|\}$$ 
on $\D\setminus f^{-1}(H_\infty)$.  Since the right-hand side is bounded above on $\D$ by assumption, $u$ is bounded above on $\D\setminus f^{-1}(H_\infty)$, so $u$ extends to a subharmonic function on $\D$ with nontangential limit $0$ at almost every point of $\T$.  Since $u$ is bounded above, we conclude that $u\leq 0$ on $\D$; in particular, $u(0)\leq 0$.
\end{proof}

By this result, the class $\mathcal C\subset\O(\D,\P^n)$ most likely to give $V_X=E_\mathcal C J$ for a nonempty bounded open subset $X$ of $\C^n$, that is, the largest reasonable class with $V_X\leq E_\mathcal C J$, is the class of analytic discs $f$ in $\P^n$ with $f(0)\in\C^n$ that have a holomorphic lifting $\tilde f=(f_0,\dots,f_n):\D\to\C_*^{n+1}$ such that $f_0$ is a Blaschke product, $f_1,\dots,f_n$ are bounded, and the boundary map $f^*=\pi\circ\tilde f^*$ takes a conull subset of $\T$ into $X$.

We do not know whether $V_X=E_\mathcal C J$.  The obstacle, as far as the present work is concerned, is that we see no way to verify or enforce condition (ii) for the disc $g$ constructed in the proof of Lemma \ref{mainlemma} below.  As far as we can tell, $g$ might well fail to satisfy (ii).  Besides, condition (ii) seems awkward and unnatural.  Rather than pursue this line of attack, we are led to a new disc functional $I$, larger than $J$, that agrees with $J$ on closed analytic discs but is more natural and easier to work with for discs with weak boundary regularity.  The main result of this paper is a disc formula for the Siciak-Zahariuta extremal function as the envelope of $I$ with respect to a very simply-defined extension of the class $\A_{\P^n}^X$.

Again, let $X$ be an open subset of $\C^n$.  We denote by $\B_{\P^n}^X$ the set of analytic discs $f$ in $\P^n$ with $f(0)\in\C^n$ that have a bounded holomorphic lifting $\tilde f:\D\to\C_*^{n+1}$ such that the boundary map $f^*=\pi\circ\tilde f^*$ takes a conull subset of $\T$ into $X$.

Since the components of $\tilde f$ are bounded, the boundary map $\tilde f^*$ exists, and the property that $\pi\circ\tilde f^*$ take a conull subset of $\T$ into $X$ does not depend on the choice of $\tilde f$: if one lifting has this property, then they all do.  Clearly, $J(f)<\infty$ for $f\in\B_{\P^n}^X$.  Also, $\A_{\P^n}^X\subset \B_{\P^n}^X$.  If $f\in\B_{\P^n}^X$ and $\psi:\D\to\D$ is a proper holomorphic map, for example an automorphism, with $f(\psi(0))\in\C^n$, then $f\circ\psi\in\B_{\P^n}^X$.

We now define a functional $I$ on the set of all analytic discs $f$ in $\P^n$ with $f(0)\in\C^n$ that have a bounded holomorphic lifting $\tilde f$.  The subharmonic function $\log\|\tilde f\|:\D\to\R$ is bounded above, so it has a Riesz decomposition into a negative potential, a negative singular harmonic function $\log|s|$, where $s$ is a singular holomorphic function, and an absolutely continuous harmonic function $\log|h|$, bounded above, where $h$ is a bounded outer function.  Dividing $\tilde f$ by the nowhere-vanishing holomorphic function $sh$, we obtain a bounded holomorphic lifting $\hat f=(f_0,\dots,f_n)$ of $f$, uniquely determined up to multiplication by a constant in $\T$ by the property that $\log\|\hat f\|$ is a potential.  The bounded holomorphic liftings of $f$ are precisely the products of $\hat f$ and the nowhere-vanishing bounded holomorphic functions on $\D$.  Factor $f_0$ as $Bsh$, where $B$ is a Blaschke product, $s$ is a singular holomorphic function, and $h$ is a bounded outer function.  Thus, $(Bs)h$ is the inner-outer factorization of $f_0$.  Define
$$I(f)=-\log|(Bs)(0)| = \int_\T\lambda\circ\hat f^*\,d\sigma - \lambda(\hat f(0)) \in [0,\infty).$$
We have
$$J(f)=-\log|B(0)|\leq I(f).$$
Clearly, $I(f)=J(f)$ if and only if $f_0$ is the product of a Blaschke product and an outer function.  In particular, this holds if $f$ extends continuously to $\overline\D$.

Note that
$$I(f)=\min_{\tilde f} \int_\T\lambda\circ\tilde f^*\,d\sigma - \lambda(\tilde f(0)),$$
where $\tilde f$ runs through all bounded holomorphic liftings of $f$, since, as noted above, $\tilde f=sh\hat f$, where $s$ is singular and $h$ is outer and bounded.  The minimum is realized precisely by the outer multiples of $\hat f$.

\begin{proposition}
\label{inequality}
Let $X$ be an open subset of $\C^n$.  Then
$$V_X\leq E_{\B_{\P^n}^X}I.$$
\end{proposition}

If $X$ is connected, then $E_{\B_{\P^n}^X}I \leq E_{\A_{\P^n}^X}I = E_{\A_{\P^n}^X}J=V_X$ by \cite{LarussonSigurdsson1}, so $V_X = E_{\B_{\P^n}^X}I$.

\begin{proof}  Let $\tilde f=(f_0,\dots,f_n)$ be a bounded holomorphic lifting of $f\in\B_{\P^n}^X$ such that $I(f)=-\log|g(0)|$, where $f_0=gh$ with $g$ inner and $h$ outer.  Define a plurisubharmonic function $v$ on $\C^{n+1}$ by the formula
$$v(z_0,\dots,z_n)=V_X(z_1/z_0,\dots,z_n/z_0)+\lambda(z)$$
for $z_0\neq 0$.  Thus defined, $v$ is locally bounded above where $z_0=0$, so $v$ extends plurisubharmonically to $\C^{n+1}$.  Namely, there is $c\in\R$ with $V_X\leq \log^+\|\cdot\|+c$ on $\C^n$, so
\begin{align*}
v(z_0,\dots,z_n) &\leq \log^+\frac {\|(z_1,\dots,z_n)\|}{|z_0|} + \log|z_0|+c \\ &=\log\max\{|z_0|, \|(z_1,\dots,z_n)\|\}+c
\end{align*}
where $z_0\neq 0$ and hence on $\C^{n+1}$.  Since $\tilde f$ is bounded, the subharmonic function $v\circ\tilde f$ is bounded above on $\D$, so it does not have a positive singular term in its Riesz decomposition, and
$$v(\tilde f(0)) \leq \int_{\T}(v\circ\tilde f)^*\,d\sigma.$$
For almost every $\zeta\in\T$, as $\eta\in\D$ approaches $\zeta$ nontangentially, $f(\eta)\in X$ for $\eta$ close enough to $\zeta$, so $(v\circ\tilde f)(\eta)\to\lambda(\tilde f^*(\zeta))=\log|h^*(\zeta)|$.  Hence,
$$v(\tilde f(0))\leq \int_{\T}(v\circ\tilde f)^*\,d\sigma = \int_\T \log|h^*|\,d\sigma = \log|h(0)|.$$
Finally,
$$V_X(f(0))-I(f)=V_X(f(0))+\log|g(0)| = v(\tilde f(0))-\log|h(0)| \leq 0,$$
so $V_X(f(0))\leq I(f)$.
\end{proof}

We conclude this section by showing that Proposition \ref{inequality} may fail if $I$ is replaced by $J$.  Take $X=D_r\subset\C$ with $1<r<e$ and $a\in(0,1)$.  Let $f=\dfrac 1{gs}\in\O(\D,\P^1)$, where $g(z)=\dfrac{z-a}{1-az}$ is an automorphism of $\D$, and $s(z)=\exp\dfrac{z+1}{z-1}$ is the singular function corresponding to the unit mass at 1.  Then $f^{-1}(\infty)=\{a\}$, so $J(f)=-\log a$.  The lifting $(gs, 1)$ of $f$ is bounded.  Now $s$ is a universal covering map $\D\to\D\setminus\{0\}$, so $s$ has nontangential boundary values in $\T\subset X$ almost everywhere on $\T$, so $f$ does as well.  Thus, $f\in\B_{\P^1}^X$ and
$$V_X(f(0)) = \log^+|f(0)/r| = \log\frac e{ar} > -\log a = J(f).$$
This shows that for a domain $X$ as simple as a disc in the plane, $V_X\neq E_{\B_{\P^n}^X}J$.

\section{A disc formula for the Siciak-Zahariuta extremal function}

\noindent
This section contains the proof of our first disc formula.

\begin{theorem}  
\label{maintheorem}
If $X$ is an open subset of $\C^n$, then
$$V_X=E_{\B_{\P^n}^X}I.$$
\end{theorem}

By Proposition \ref{inequality}, we need to show that $V_X\geq E_{\B_{\P^n}^X}I$.  The proof is based on two lemmas about an auxiliary class of analytic discs, defined as follows.  Let $X$ be an open subset of $\C^n$.  Let $\mathcal G_{\P^n}^X$ be the subset of $\A_{\P^n}^X$ of all closed analytic discs $f$ in $\P^n$ with $f(0)\in\C^n$ and $f(\T)\subset X$ that lift to a closed analytic disc $\tilde f$ in $\C_*^{n+1}$ with
$$\max_\T\|\tilde f\| < 2\min_\T\|\tilde f\|.$$
(The only property of the number $2$ that matters here is that $2>1$.)

\begin{lemma}  
\label{pshminorantofenv}
Let $X$ be a nonempty open subset of $\C^n$.  The envelope $E_{\mathcal G_{\P^n}^X} J$ is upper semicontinuous on $\C^n$.  It has minimal growth on $\C^n$ and vanishes on $X$, so its largest plurisubharmonic minorant is no larger than the Siciak-Zahariuta extremal function $V_X$ of $X$.
\end{lemma}

\begin{proof}  Take $c\in X$, $c\neq 0$, and choose $r>0$ small enough that
$$\frac{1+(\|c\|+r)^2}{1+(\|c\|-r)^2}<4,$$
and that the closed ball $B$ with centre $c$ and radius $r>0$ is contained in $X$.

For $z\in X$, the constant disc at $z$ is in $\mathcal G_{\P^n}^X$, so $E_{\mathcal G_{\P^n}^X} J(z)=0$.  For $z\in\C^n\setminus B$, let $f=g\circ\phi\in\O(\overline \D, \P^n)$, where 
$$g(\zeta)=\frac r{\zeta\|z-c\|}(z-c)+c,$$
and $\phi$ is an automorphism of $\D$ interchanging $0$ and $g^{-1}(z)=r/\|z-c\|$.  Then $f(\T)\subset\partial B\subset X$.  The lifting $\tilde f=\tilde g\circ\phi$ of $f$, where 
$$\tilde g(\zeta)= (\zeta, \frac r{\|z-c\|}(z-c)+\zeta c),$$ 
satisfies
$$\frac{\max_\T\|\tilde f\|^2}{\min_\T\|\tilde f\|^2} \leq \frac{1+(\|c\|+r)^2}{1+(\|c\|-r)^2}<4,$$
so $f\in\mathcal G_{\P^n}^X$.  Also $f(0)=z$ and $f^{-1}(H_\infty)=\{r/\|z-c\|\}$, so 
$$E_{\mathcal G_{\P^n}^X} J(z)\leq J(f)=\log\|z-c\|-\log r.$$
This proves that $E_{\mathcal G_{\P^n}^X} J$ has minimal growth.

We have shown that for every $z\in\C^n$, there is $f\in\mathcal G_{\P^n}^X$ with $f(0)=z$.  For translations $\tau$ of $\C^n$ sufficiently close to the identity, $f_\tau=\tau\circ f\in\mathcal G_{\P^n}^X$, the map $\tau\mapsto J(f_\tau)$ is constant, and the centres $f_\tau(0)=\tau(z)$ sweep out a neighbourhood of $z$.  Hence, $E_{\mathcal G_{\P^n}^X} J$ is upper semicontinuous at $z$.  
\end{proof}

\begin{lemma}  
\label{mainlemma}
Let $X$ be an open subset of $\C^n$.  Let $Y$ be a nonempty relatively compact open subset of $X$.  For every closed analytic disc $h$ in $\C^n$, continuous function $v\geq E_{\mathcal G_{\P^n}^Y} J$ on $\C^n$, and $\epsilon>0$, there is $g\in\B_{\P^n}^X$ with $g(0)=h(0)$ and
$$I(g) < \int_\T v\circ h\,d\sigma +\epsilon.$$
\end{lemma}

This lemma is related to the fundamental theorem of the theory of disc functionals, which states that the largest plurisubharmonic minorant of an upper semicontinuous function $u$ on a domain $\Omega$ in $\C^n$ is the envelope of the Poisson functional $f\mapsto\int_\T u\circ f\,d\sigma$, $f\in\O(\overline\D,\Omega)$.  In \cite{LarussonSigurdsson1}, a similar lemma was established along the lines of Poletsky's original proof of the fundamental theorem \cite{Poletsky1}.  The present authors spent a great deal of time unsuccessfully trying to give a similar proof of Lemma \ref{mainlemma}.  The proof below is based on Bu and Schachermayer's alternative approach to the fundamental theorem \cite{BuSchachermayer}.  Their method is better suited to analytic discs with weak boundary regularity.

Before proving Lemma \ref{mainlemma}, let us quickly prove Theorem \ref{maintheorem}.

\begin{proof}[Proof of Theorem \ref{maintheorem}]  Let $Y$ be a nonempty relatively compact open subset of $X$.  Fixing $z\in \C^n$, and taking the infimum in Lemma \ref{mainlemma} over all $v$, $\epsilon$, and $h$ with $h(0)=z$, we see that $E_{\B_{\P^n}^X}I$ is no larger than the Poisson envelope, that is, the largest plurisubharmonic minorant, of $E_{\mathcal G_{\P^n}^Y}J$ on $\C^n$.  By Lemma \ref{pshminorantofenv}, this envelope is no larger than $V_Y$.  Thus, by Proposition \ref{inequality},
$$V_X \leq E_{\B_{\P^n}^X}I \leq V_Y.$$
Finally, the infimum of $V_Y$ over all relatively compact open subsets $Y$ of $X$ is $V_X$. 
\end{proof}

\begin{proof}[Proof of Lemma \ref{mainlemma}]  Let $B$ be the closed ball in $\C^{n+1}$ of radius $\frac 1 2 \sqrt{1+\min_\T\|h\|^2}$, centred at the origin.  There is $\delta>0$ so small that: 
\begin{enumerate}
\item[(a)]  $\pi^{-1}(X)$ contains the $\delta$-neighbourhood of $\pi^{-1}(Y)\setminus B$.
\item[(b)]  $\lambda$ is bounded below on the $\delta$-neighbourhood of $\pi^{-1}(Y)\setminus B$, so $\lambda$ is uniformly continuous there.
\item[(c)]  $|\lambda(z+w)-\lambda(z)|<\epsilon/4$ if $z\in\pi^{-1}(Y)\setminus B$ and $\|w\|<\delta$.
\end{enumerate}

For each $\eta\in\T$, find $f\in \mathcal G_{\P^n}^Y$ such that $f(0)=h(\eta)$ and $J(f)<E_{\mathcal G_{\P^n}^Y}J(h(\eta))+\epsilon/2$.  There is an open arc $A\subset\T$ containing $\eta$ such that $J(f)<v\circ h+\epsilon/2$ (here we need lower semicontinuity of $v$) and $\|h-h(\eta)\|<\delta/2$ on $A$.  We can cover $\T$ by open arcs of this kind and then pass to finitely many nondegenerate compact subarcs $A_1,\dots,A_k$ that cover $\T$ and are mutually disjoint except for common endpoints.  For $j=1,\dots,k$, let $f_j\in\mathcal G_{\P^n}^Y$ and $\eta_j\in A_j$ satisfy $f_j(0)=h(\eta_j)$, $J(f_j)<v\circ h+\epsilon/2$, and $\|h-h(\eta_j)\|<\delta/2$ on $A_j$.  Setting $a_j=\sigma(A_j)\in(0,1)$, we have
$$\sum_{j=1}^k a_j J(f_j) = \sum_{j=1}^k\int_{A_j}J(f_j)\,d\sigma<\int_\T v\circ h\,d\sigma+\epsilon/2.$$
Let $\tilde h$ be the holomorphic lifting $(1,h):\overline\D\to\C_*^{n+1}$ of $h$.  Let $\tilde f_j:\overline \D\to\C_*^{n+1}$ be a holomorphic lifting of $f_j$ with $\max_\T\|\tilde f_j\| < 2\min_\T\|\tilde f_j\|$ and $\tilde f_j(0)=\tilde h(\eta_j)$.  Then
$$\min_\T\|\tilde f_j\|>\tfrac 1 2\max_\T\|\tilde f_j\| \geq \tfrac 1 2 \|\tilde f_j(0)\| = \tfrac 1 2 \|\tilde h(\eta_j)\| \geq \tfrac 1 2 \sqrt{1+\min_\T\|h\|^2},$$
so $\tilde f_j(\T)\subset\pi^{-1}(Y)\setminus B$.

As in the proof of Lemma III.2 in \cite{BuSchachermayer}, for $j=1,\dots,k$, we can find a sequence of holomorphic functions $\alpha_{jm}:\D\to\D$, $m\in\mathbb N$, such that:
\begin{enumerate}
\item  $\alpha_{jm}(0)=e^{-m(1-a_j)}$.
\item  $|\alpha_{jm}^*|=1$ a.e.\ on $A_j$.
\item  $|\alpha_{jm}^*|=e^{-m}$ a.e.\ on $\T\setminus A_j$.
\item  $(\alpha_{jm}^*)_*\sigma|A_j \to a_j\sigma$ weakly on $\T$ as $m\to\infty$.
\end{enumerate}
The function $\alpha_{jm}$ is simply the outer function $\exp(s+it)$, where the harmonic function $s$ is the Poisson integral of the function that equals $0$ on $A_j$ and $-m$ on $\T\setminus A_j$, and the conjugate function $t$ is normalized by $t(0)=0$.  Only property (4) is nontrivial; for a proof, see \cite{BuSchachermayer}, p.\ 594.

Following the proof of Proposition III.3 in \cite{BuSchachermayer}, for each $m\in\mathbb N$, we define a holomorphic map $\tilde g_m:\D\to\C^{n+1}$ by the formula
$$\tilde g_m(z)=\tilde h(z)+\sum_{j=1}^k \big(\tilde f_j\circ\alpha_{jm}(z)-\tilde f_j\circ\alpha_{jm}(0)\big).$$
Note that $\tilde g_m$ is bounded and $\tilde g_m(0)=\tilde h(0)$.

We claim that for $m$ sufficiently large, the boundary map $\tilde g_m^*$ takes a conull subset of $\T$ into $\pi^{-1}(X)$.  There is $m_1\in\mathbb N$ such that for every $m\geq m_1$,
$$\|\tilde f_j(0)-\tilde f_j(e^{-m(1-a_j)})\|<\delta/4$$ 
and 
$$\|\tilde f_j-\tilde f_j(e^{-m(1-a_j)})\|< \dfrac \delta{4(k-1)} \quad\text{on } T_{e^{-m}}$$
for $j=1,\dots,k$.  Let $m\geq m_1$ and $1\leq j\leq k$.  For almost every $\zeta\in A_j$, $\alpha_{j m}^*(\zeta)$ exists and lies in $\T$, and, for each $\nu\neq j$, $\alpha_{\nu m}^*(\zeta)$ exists and lies in $T_{e^{-m}}$.  We have
\begin{align*}
\tilde g_m^*(\zeta)=\tilde f_j(\alpha_{j m}^*(\zeta)) &+ \big(\tilde h(\zeta)-\tilde h(\eta_j)\big) + \big(\tilde f_j(0)-\tilde f_j(e^{-m(1-a_j)})\big) \\ &+ \sum_{\nu\neq j} \big(\tilde f_\nu(\alpha_{\nu m}^*(\zeta))-\tilde f_\nu(e^{-m(1-a_\nu)})\big).
\end{align*}
Since $\|\tilde h(\zeta)-\tilde h(\eta_j)\|<\delta/2$, we see that $\tilde g_m^*(\zeta)$ lies within $\delta$ of $\tilde f_j(\T)$, so $\tilde g_m^*(\zeta)\in\pi^{-1}(X)$.

Next we claim that for $m$ sufficiently large, 
$$\int_\T\lambda\circ\tilde g_m^*\,d\sigma < \sum_{j=1}^k a_j J(f_j)+\epsilon/2.$$
Since the zeroth component of $\tilde f_j(0)=\tilde h(\eta_j)$ is $1$,
$$J(f_j)=\int_\T \lambda\circ\tilde f_j\,d\sigma.$$
On each arc $A_j$, 
$$\lambda\circ\tilde g_m^*=\lambda\circ(\tilde f_j\circ\alpha_{j m}^*+\xi) \text{ a.e.},$$
where
$$\xi=\big(\tilde h-\tilde h(\eta_j)\big) + \big(\tilde f_j(0)-\tilde f_j(e^{-m(1-a_j)})\big)+\sum_{\nu\neq j} \big(\tilde f_\nu\circ\alpha_{\nu m}^*-\tilde f_\nu (e^{-m(1-a_\nu)})\big).$$
For $m\geq m_1$, $\|\xi\|<\delta$ a.e.\ on $A_j$, so 
$$\lambda\circ\tilde g_m^* \leq \lambda\circ\tilde f_j\circ\alpha_{j m}^*+\epsilon/4\quad\text{a.e.\ on }A_j.$$
By (4) above,
$$\int_{A_j}\lambda\circ\tilde f_j\circ\alpha_{j m}^*\,d\sigma = \int_\T \lambda\circ\tilde f_j\,d(\alpha_{j m}^*)_*\sigma|A_j \to a_j\int_\T \lambda\circ\tilde f_j\,d\sigma = a_j J(f_j)$$
as $m\to\infty$.  Find $m_2\in\mathbb N$ such that 
$$\int_{A_j}\lambda\circ\tilde f_j\circ\alpha_{j m}^*\,d\sigma < a_j (J(f_j) + \epsilon/4)$$ 
for all $m\geq m_2$ and $j=1,\dots,k$.  We conclude that for $m\geq\max\{m_1,m_2\}$,
\begin{align*}
\int_\T\lambda\circ\tilde g_m^*\,d\sigma &\leq \sum_{j=1}^k \int_{A_j}\big(\lambda\circ\tilde f_j\circ\alpha_{jm}^*+\epsilon/4\big)\,d\sigma \\ &< \sum_{j=1}^k \big(a_j(J(f_j)+\epsilon/4)+a_j\epsilon/4\big) = \sum_{j=1}^k a_jJ(f_j)+\epsilon/2.
\end{align*}

For every $m\in\mathbb N$, $\tilde g_m^{-1}(0)$ is a discrete subset of $\D$ satisfying the Blaschke condition, so the harmonic measure of the complement $\D\setminus\tilde g_m^{-1}(0)$ with respect to the origin is simply $\sigma$.  Let $\phi_m:\D\to\D\setminus\tilde g_m^{-1}(0)$ be a holomorphic covering map with $\phi_m(0)=0$.  Then $\phi_m$ takes a conull subset of $\T$ into $\T$, and the measurable map $\phi_m^*:\T\to\T$ satisfies $(\phi_m^*)_*\sigma=\sigma$.

Let $\check g_m$ be the bounded analytic disc $\tilde g_m\circ\phi_m:\D\to\C_*^{n+1}$ with $\check g_m(0)=\tilde h(0)$.  We have shown that for $m$ sufficiently large, $\check g_m$ takes a conull subset of $\T$ into $\pi^{-1}(X)$, so $g_m=\pi\circ\check g_m\in \B_{\P^n}^X$, and
$$I(g_m)\leq \int_\T\lambda\circ\check g_m^*\,d\sigma = \int_\T\lambda\circ\tilde g_m^*\,d\sigma< \sum_{j=1}^k a_jJ(f_j)+\epsilon/2<\int_\T v\circ h\,d\sigma+\epsilon.$$
Taking $g=g_m$ for $m$ sufficiently large completes the proof.
\end{proof}

\section{A second disc formula and a characterization of polynomial hulls}

\noindent
An analytic disc $f:\D\to\C^n$ is said to be Nevanlinna if each of its components is a Nevanlinna function or, equivalently, the subharmonic function $\log\|f\|$ has a positive harmonic majorant on $\D$.  A Nevanlinna function can be factored, essentially uniquely, as $Bhs/t$, where $B$ is a Blaschke product, $h$ is an outer function, and $s$, $t$ are singular functions given by mutually singular measures.  A singular function $s$ is of the form
$$s(z)=\exp\int_\T \frac{z+\zeta}{z-\zeta}\,d\mu(\zeta),$$
where $\mu$ is a finite positive Borel measure on $\T$ that is singular with respect to $\sigma$ and uniquely determined by $s$.  We have $s(0)=e^{-\mu(\T)}$.  Customarily, $N$ denotes the class of all Nevanlinna functions, and $N^+$ denotes the class of Nevanlinna functions with $t=1$.  Let us introduce the notation $N^-$ for the class of Nevanlinna functions with $s=1$.

Let $f=(f_1,\dots,f_n):\D\to\C^n$ be a Nevanlinna disc.  Let the singular denominator $t_j$ of $f_j$ be given by the measure $\nu_j$.  Let $\nu$ be the smallest Borel measure on $\T$ with $\nu\geq\nu_j$ for $j=1,\dots,n$.  It is the measure corresponding to the singular function $t$ that is the least common multiple of $t_1,\dots,t_n$ with respect to the order defined by divisibility on the set of singular functions.  We call 
$$\nu(f)=\nu(\T)=-\log t(0)\in[0,\infty)$$
the negative mass of $f$.  We have $\nu(f)=0$ if and only if $f_1,\dots,f_n\in N^+$.

Let $X$ be an open subset of $\C^n$.  We denote by $\mathcal N_X$ the set of Nevanlinna discs in $\C^n$ whose boundary map takes a conull subset of $\T$ into $X$, and by $\mathcal N_X^-\subset \mathcal N_X$ the subset of discs whose components lie in $N^-$.

\begin{proposition}
\label{nu=I}
{\rm (a)}  Let $X$ be an open subset of $\C^n$.  Then $\mathcal N_X\subset\B_{\P^n}^X$. \newline
{\rm (b)}  Let $f$ be a Nevanlinna disc in $\C^n$.  Then $\nu(f)=I(f)$.
\end{proposition}

\begin{proof}  (a)  This follows easily from the fact that every Nevanlinna function is the quotient of a bounded holomorphic function by a nowhere-vanishing bounded holomorphic function.

(b)  Let $f=(f_1,\dots,f_n)$ be a Nevanlinna disc in $\C^n$.  For $j=1,\dots,n$, write $f_j=B_j h_j s_j/t_j$, where $B_j$ is Blaschke, the singular functions $s_j$, $t_j$ are given by mutually singular measures, and $h_j=u_j/v_j$, where $u_j$, $v_j$ are bounded and outer.  Let $t$ be the least common multiple of $t_1,\dots,t_n$, say $t=r_j t_j$, where $r_j$ is singular.  Then $\nu(f)=-\log t(0)$, and
$$(t v_1\cdots v_n, B_1 u_1 v_2\cdots v_n r_1 s_1 , \dots, B_n u_n v_1\cdots v_{n-1} r_n s_n):\D\to\C_*^{n+1}$$
is a bounded holomorphic lifting of $f$.

Let $\tilde g=(g_0,\dots,g_n):\D\to\C_*^{n+1}$ be a bounded holomorphic lifting of $f$.  For $j=0,\dots,n$, write $g_j=B'_j h'_j s'_j$, where $B'_j$ is Blaschke, $h'_j$ is bounded and outer, and $s'_j$ is singular.  Of course $B'_0=1$.  For $j=1,\dots,n$, we have
$$\frac{B'_j h'_j s'_j}{h'_0 s'_0} = \frac{g_j}{g_0} = f_j = \frac{B_j h_j s_j}{t_j}.$$
Since $s_j$ and $t_j$ are mutually singular, $t_j$ divides $s'_0$ (and $s_j$ divides $s'_j$) for $j=1,\dots,n$, so $t$ divides $s'_0$ and $-\log t(0)\leq -\log s'_0(0)$.

By definition, $I(f)$ is the minimum of $-\log s'_0(0)$ over all bounded holomorphic liftings $\tilde g$ of $f$ as above.  It follows that $\nu(f)=I(f)$.
\end{proof}

Our second disc formula for the Siciak-Zahariuta extremal function is a corollary of the disc formula in Theorem \ref{maintheorem}.

\begin{theorem}  
\label{secondformula}
If $X$ is an open subset of $\C^n$, then
$$V_X=E_{\mathcal N_X}\nu=E_{\mathcal N_X^-}\nu.$$
\end{theorem}

\begin{proof}  By Proposition \ref{nu=I}, $E_{\B_{\P^n}^X}I\leq E_{\mathcal N_X}\nu$.  To prove the opposite inequality, take $f\in\B_{\P^n}^X$.  Let $\tilde f=(f_0,\dots,f_n):\D\to\C_*^{n+1}$ be a bounded holomorphic lifting of $f$ with $I(f)=-\log|i(0)|$, where $i$ is the inner factor of $f_0$.  As in the proof of Lemma \ref{mainlemma}, we note that the discrete subset $f_0^{-1}(0)=f^{-1}(H_\infty)$ of $\D$ satisfies the Blaschke condition, so the harmonic measure of the complement $\D\setminus f^{-1}(H_\infty)$ with respect to the origin is $\sigma$.  Let $\phi:\D\to\D\setminus f^{-1}(H_\infty)$ be a holomorphic covering map with $\phi(0)=0$.  It takes a conull subset of $\T$ into $\T$, and its boundary map preserves $\sigma$.  Therefore, $g=f\circ\phi\in\mathcal N_X$, $g(0)=f(0)$, and, by Proposition \ref{nu=I}, 
$$\nu(g)=I(g) \leq \int_\T\lambda\circ\tilde g^*\,d\sigma - \lambda(\tilde g(0))= \int_\T\lambda\circ\tilde f^*\,d\sigma - \lambda(\tilde f(0))=I(f),$$
where $\tilde g$ is the lifting $\tilde f\circ\phi$ of $g$.  By Theorem \ref{maintheorem}, this shows that $V_X=E_{\mathcal N_X}\nu$.

As $Y$ runs through all relatively compact open subsets of $X$, the infimum of $V_Y=E_{\mathcal N_Y}\nu$ is $V_X=E_{\mathcal N_X}\nu$.  It follows that $V_X$ is the envelope of $\nu$ with respect to the class of discs in $\mathcal N_X$ that take a conull subset of $\T$ into a compact subset of $X$.

To show that $E_{\mathcal N_X}\nu=E_{\mathcal N_X^-}\nu$, take $f\in\mathcal N_X$.  As just noted, we may assume that $f^*$ takes a conull subset of $\T$ into a compact subset $K$ of $X$.  Find $\epsilon>0$ such that $(c_1y_1,\dots,c_ny_n)\in X$ if $(y_1,\dots,y_n)\in K$ and $c_j\in \C$, $|c_j-1|<\epsilon$, for $j=1,\dots,n$.

For $j=1,\dots,n$, write $f_j=B_j h_j s_j/t_j$, where $B_j$ is Blaschke, $h_j$ is outer, and the singular functions $s_j$, $t_j$ are given by mutually singular measures.  For $a\in\D$, let $\psi_a$ be the automorphism $z\mapsto\dfrac{z-a}{1-\bar a z}$ of $\D$.  By a theorem of Frostman (\cite{Garnett}, Theorem 6.4), if $k$ is a nonconstant inner function, then $\psi_a\circ k$ is a Blaschke product for almost every $a\in\D$ (in fact outside a set of capacity zero).  Find $a\in \D$ close enough to $0$ that
$$\bigg|\frac{s_j(0)}{\psi_a(s_j(0))} \frac{\psi_a(\zeta)}\zeta - 1 \bigg|<\epsilon$$
for all $\zeta\in\T$, and such that $\psi_a\circ s_j$ is a Blaschke product (or the constant 1), for $j=1,\dots,n$.  Define
$$g_j=\big[B_j (\psi_a\circ s_j)\big]\bigg[\frac{s_j(0)}{\psi_a(s_j(0))}h_j\bigg]/t_j.$$
Then $g=(g_1,\dots,g_n)\in\mathcal N_X^-$.  Clearly, $g(0)=f(0)$ and $\nu(g)=\nu(f)$.
\end{proof}

A description of the polynomial hull of an arbitrary compact subset of $\C^n$ in terms of Nevanlinna discs is now immediate.

\begin{corollary}
\label{hulls}
Let $K$ be a compact subset of $\C^n$.  For $a\in\C^n$, the following are equivalent.
\begin{enumerate}
\item[(i)]  $a$ is in the polynomial hull $\hat K$ of $K$.
\item[(ii)]  For every neighbourhood $U$ of $K$ and every $\epsilon>0$, there is $f\in\mathcal N_U^-$ with $f(0)=a$ and $\nu(f)<\epsilon$.
\item[(iii)]  For every neighbourhood $U$ of $K$ and every $\epsilon>0$, there is $f\in\mathcal N_U$ with $f(0)=a$ and $\nu(f)<\epsilon$.
\end{enumerate}
\end{corollary}

It may shed light on this result to note that for every nonempty open subset $U$ of $\C^n$ and every $a\in\C^n$, there is $f\in\mathcal N_U^-$ with $f(0)=a$.  The question is how small the negative mass of $f$ can be.

In \cite{LarussonSigurdsson2}, we proved the equivalence of (i) and (ii) under the assumption that $K$ is connected.

\begin{proof}    We have $a\in\hat K$ if and only if $V_K(a)=0$, where $V_K$ is the unregularized Siciak-Zahariuta extremal function of $K$.  Also, $V_K=\sup V_U$, where $U$ runs through any basis of neighbourhoods of $K$ in $\C^n$.  Hence, $a\in\hat K$ if and only if $V_U(a)=0$ for every open neighbourhood $U$ of $K$.  By Theorem \ref{secondformula}, this means that for every open neighbourhood $U$ of $K$ and every $\epsilon>0$, there is $f\in\mathcal N_U$, or equivalently $f\in\mathcal N_U^-$, with $f(0)=a$ and $\nu(f)<\epsilon$.
\end{proof}

\end{document}